\documentclass[preprint,12pt]{elsarticle}
 
\usepackage{etoolbox}
\makeatletter
\def\ps@pprintTitle{%
	\let\@oddhead\@empty
	\let\@evenhead\@empty
	\def\@oddfoot{\reset@font\hfil\thepage\hfil}
	\let\@evenfoot\@oddfoot
}
\makeatother
\usepackage{bm}
\usepackage{float}
\usepackage{graphics}
\usepackage{amsmath, amsthm, amssymb}
\usepackage{natbib}
\usepackage{caption}
\usepackage[colorlinks=true]{hyperref}
\usepackage{lscape}
\usepackage{float}
\usepackage[utf8]{inputenc} 
\usepackage{amsthm}

\usepackage{enumerate}
\usepackage{mathrsfs}
\usepackage{placeins}
\usepackage[inner=4cm,outer=2cm]{geometry}
\usepackage{graphicx}
\usepackage{amsfonts}
\usepackage{verbatim}
\usepackage{amssymb}
\usepackage{amsthm,multirow}
\usepackage{orcidlink}

\newtheorem{thm}{Theorem}[section]
\newtheorem{lem}{Lemma}[section]
\newtheorem{cor}{Corollary}[section]
\newtheorem{prop}{Proposition}[section]

\theoremstyle{definition}
\newtheorem{defn}{Definition}[section]
\newtheorem{ex}{Example}[section]

\theoremstyle{remark}

\theoremstyle{properties}

\theoremstyle{Examples}

\numberwithin{equation}{section}

\usepackage[mathscr]{euscript}
\usepackage{geometry} 
\usepackage{graphicx,lscape} 
\usepackage{booktabs} 
\usepackage{bigstrut}
\usepackage{array} 
\usepackage{paralist} 
\usepackage{verbatim} 
\usepackage{subfig} 
\biboptions{comma,round,authoryear}
\begin{document}
\begin{frontmatter}
	\title{\textbf{Fractional Cumulative Residual Entropy \\in the Quantile Framework and its Applications \\in the Financial Data}}
	\author{Iona Ann Sebastian \orcidlink{0009-0000-7640-2215}\corref{cor1}, S.M.Sunoj \orcidlink{0000-0002-6227-1506}\corref{cor1}}
	\ead{ionaann99@gmail.com, smsunoj@cusat.ac.in}
	\cortext[cor1]{Corresponding author}
	
	\address{Department of Statistics\\Cochin University of Science and Technology\\Cochin 682 022, Kerala, India}

	\begin{abstract}
	Fractional cumulative residual entropy (FCRE) is a powerful tool for the analysis of complex systems. Most of the theoretical results and applications related to the FCRE of the lifetime random variable are based on the distribution function approach. However, there are situations in which the distribution function may not be available in explicit form but has a closed-form quantile function (QF), an alternative method of representing a probability distribution.  Motivated by this, in the present study we introduce a quantile-based FCRE, its dynamic version and their various properties and examine their usefulness in different applied fields. 
	\end{abstract}
	
	
	
\end{frontmatter}

\section{Introduction}
It was Rudolph Clausius, a physicist and a mathematician, who coined the term entropy in the context of identifying the lost energy in the thermodynamic system. However, the definition, interpretation, and properties of entropy as a measure of uncertainty that formed the basis of information theory was due to \citet{shannon1948mathematical}, defined for a discrete random variable (rv).  The notion of entropy corresponding to a continuous rv is known as differential entropy.  When $X$ is a non-negative absolutely continuous rv with cumulative distribution function (cdf) $F(\cdot)$ and probability density function (pdf) $f(\cdot)$, then Shannon differential entropy is defined as
\begin{equation*}
	H(X)= - \int_{0}^{\infty} f(x) \log f(x)dx,
\end{equation*}
where log stands for the natural logarithm with 0 log 0 = 0.  The average uncertainty contained in the pdf $f(x)$ of $X$ is measured by $H(X)$.  The differential entropy is not always positive. Moreover, it is difficult to estimate the differential entropy of a continuous variable through empirical distributions. To overcome these difficulties, \cite{rao2004cumulative} proposed an alternative, a more general measure of uncertainty, called cumulative residual entropy (CRE). The CRE of a non-negative continuous rv $X$
denoted by $\mathcal{E}(X)$ is defined as
\begin{equation*}
	\mathcal{E}(X)=-\int_{0}^{\infty} \bar{F}(x) \log \bar{F}(x)dx,  
\end{equation*}
where $\bar{F}(x) = 1 - F (x) = P(X > x)$ denotes the survival (reliability) function of $X$.\\

The generalization of the Shannon entropy to different fields is interesting among researchers.  \cite{ubriaco2009entropies} defined a new entropy based on fractional calculus, which is
\[
S_\alpha(P)=\sum_{i} p_i \left( - \log p_i \right)^\alpha, \quad 0 \leq \alpha \leq 1.
\]

Various generalized entropy and cumulative entropy measures are available in literature.  Based on the fractional entropy developed using the principle of fractional calculus, \cite{xiong2019fractional} proposed a fractional cumulative residual entropy (FCRE) for an absolutely continuous non-negative rv $X$ with CDF $F(\cdot)$, as
\begin{equation}\label{1}
	\mathcal{E}_{\alpha}(X) = \int_{0}^{\infty}\Bar{F}(x) \left(-\log \Bar{F}(x)\right)^\alpha dx, \; 0\leq \alpha \leq 1.
\end{equation}
It is clear that $\mathcal{E}_0(X) = E(X)$ and when $\alpha = 1$, the FCRE becomes the CRE. That is, $\mathcal{E}_1(X) = \mathcal{E} (X)$.\\

Although abundant research is available for different uncertainity measures, a study of the FCRE utilizing quantile function (QF) does not seem to have been undertaken. It has been showed by many authors that the QF defined by
\begin{equation*}
	Q(u)=F^{-1}(u)=\inf \{x|F(x)\geq u\}, \; 0 \leq u \leq 1
\end{equation*}
is an efficient and equivalent alternative to the distribution function in modelling and analysis of statistical data (see \cite{gilchrist2000statistical}, \cite{nair2009quantile}). There are several properties of QFs that are not carried out
by the distribution functions. For instance, two QFs are added together to form another QF. Since QF is less affected by outliers and offers a straightforward analysis with limited information, it is frequently more convenient.
For some recent studies on QF, 
it's properties and usefulness in the identification of models we refer to \cite{nair2013quantile}, \cite{vineshkumar2021inferring}, \cite{varkey2023review}, \cite{aswin2023reliability}, \cite{kayal2024quantile},  and the references therein. Also, many QFs used in applied works such as various forms of lambda distributions (\cite{ramberg1974approximate}; \cite{gilchrist2000statistical}), the power-Pareto distribution (\cite{gilchrist2000statistical}
; \cite{hankin2006new}), Govindarajulu distribution (\cite{nair2012modelling}) etc. do not have tractable distribution functions but have explicit form of quantile functions.
This makes it difficult to use \eqref{1} to statistically study the properties of $\mathcal{E}_{\alpha}(X)$ for these distributions. Thus, formulation
of definition and properties of FCRE based on QFs is called for. The quantile-based FCRE (Q-FCRE) has several pros. First of all,
unlike \eqref{1},  the suggested measure is much easier to compute when the distribution functions are not tractable while the QFs have simple forms. Furthermore, quantile functions have certain properties that probability distributions do not possess. 
Applying these properties yield some new results and better
insight into the measure that are not posssible to obtain by the conventional approach.\\

The paper is organized into eight sections. After the present introductory section, in Sections 2 and 3 we introduce the quantile-based fractional cumulative residual entropy (Q-FCRE), its dynamic form, and study their various properties. In Section 4 we propose a non-parametric estimator for computing Q-FCRE.  In Section 5 we carry out simulation studies to illustrate the performance of the estimator given in Section 4. A method of validation of Q-FCRE based on the simulated values on the logistic map proposed by \citet{may1976simple} is presented in Section 6.  Finally, in Sections 7 and 8 we illustrate, respectively, a potential application of Q-FCRE in a complex financial system using the Dow Jones Industrial Average (DJIA) data set and concluding remarks of the study.

\section{Quantile-based FCRE}

\begin{defn}
	\textnormal{For a nonnegative continuous rv $X$ with quantile function $Q(u)$, the quantile-based FCRE (Q-FCRE) is defined as;
		\begin{eqnarray}\label{2}
			\mathcal{E}_{\alpha}^{Q}(X)&=&\int_{0}^{1}(1-p)(-\log(1-p))^\alpha dQ(p) \nonumber\\ 
			& =&\int_{0}^{1}(1-p)(-\log(1-p))^\alpha q(p)dp, \; 0 \leq \alpha \leq 1
		\end{eqnarray}
		where $q(p)=\frac{d}{dp}Q(p)$ is the quantile density function (qdf).}
\end{defn}

It is clear that when $\alpha=1$, $\mathcal{E}_{\alpha}^{Q}(X)$ becomes quantile-based cumulative residual entropy \citep{sankaran2017quantile}. $\mathcal{E}_{\alpha}^{Q}(X)$ is non-negative and non-additive. It is a convex function of parameter $\alpha$.  In terms of hazard quantile function, denoted by $H(u)=[(1-u)q(u)]^{-1}$, \eqref{2} can be represented as 
\begin{equation}
	\mathcal{E}_{\alpha}^{Q}(X)=\int_{0}^{1}[H(p)]^{-1}(-\log(1-p))^\alpha dp. \nonumber
\end{equation}

We now provide some examples of Q-FCRE for some important probability distributions.
\begin{itemize}
	\item [(i)] If $X$ is uniformly distributed over $[0, b]$, $b>0$, with the qdf, $q(u)=b$, then
	$\mathcal{E}_{\alpha}^{Q}(X)= \frac{{b\Gamma(\alpha+1) }}{2^{\alpha+1}} $
	where $\Gamma (\cdot)$ is the complete gamma function. When $\alpha=1$, $\mathcal{E}_{Q}(X)=b/4$.
	\item [(ii)] If $X$ is following the exponential distribution with mean $1/\lambda$ and the qdf $q(u)=1/\lambda(1-u)$, then $\mathcal{E}_{\alpha}^{Q}(X) = \frac{\Gamma(\alpha+1)}{\lambda}$.  When $\alpha=1$, $\mathcal{E}_{\alpha}^{Q}(X)=1/\lambda$.

    \begin{figure}[H]
		\centering
		\includegraphics[width=0.9\linewidth]{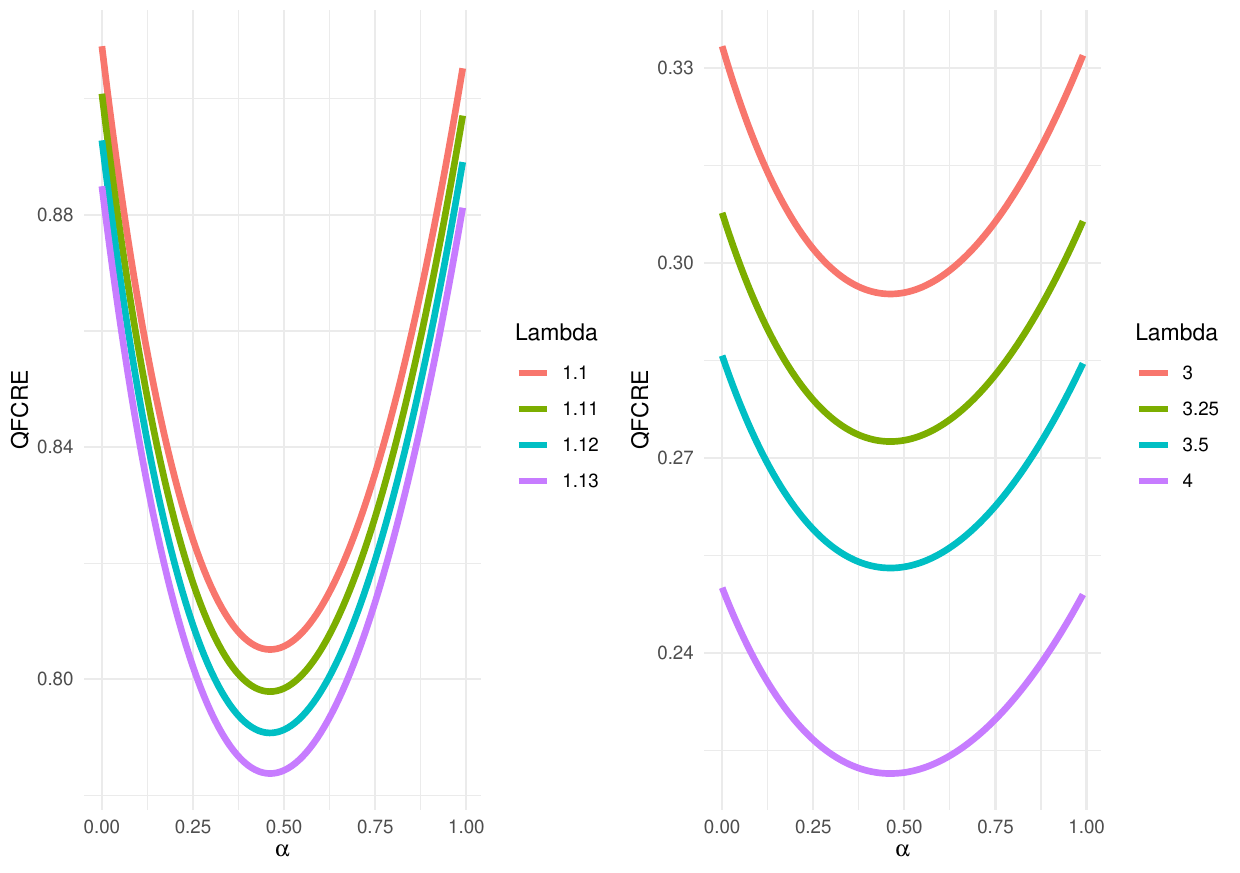}
		\caption{\textit{Plots of QFCRE for exponentially distributed rv given in Example 2, for  $\lambda =
				1.1, 1.11, 1.12, 1.13$ and $\lambda = 3, 3.25, 3.5, 4$, the values of $\alpha$ is taken along the
				x-axis.
		}}
		\label{figure 1}		
	\end{figure}
\end{itemize}
Certain models do not have closed expressions for distribution functions but have an explicit qf or qdf, which are considered in the following examples.

\begin{itemize}
	\item [(iii)] For the rv $X$ with qdf given by $q(u)=Cu^{\beta}(1-u)^{-(A+\beta)}$, where $C,A,\beta \in (- \infty, +\infty)$ (see Theorem 3.2 of \cite{nair2012modelling}), we get
	\begin{equation}
		\mathcal{E}_{\alpha}^{Q}(X)=C\int_{0}^{1}(1-p)^{-(A+\beta-1)}(-\log(1-p))^{\alpha}p^{\beta}dp. \nonumber
	\end{equation}
	When $C=2, A=0, \beta=0, \alpha=0.75$, $\mathcal{E}_{\alpha}^{Q}(X)=\frac{3 \Gamma(3/4)}{2^{11/4}}$ and
	when $C=2, \beta=0, A=0.5,\alpha=1, \mathcal{E}_{\alpha}^{Q}(X)=0.665$.  The members with this quantile function have monotone or non-monotone hazard quantile functions. Moreover, it contains several well-known distributions such as the exponential ($\beta=0, A=1$), Pareto ($\beta=0, A<1$), rescaled beta ($\beta=0, A>1$), the log-logistic distribution ($\beta=\lambda-1, A=2$) where $\beta, \lambda > 0$ and the life distribution proposed by \cite{Govindarajula1977ACO} ($\beta=\gamma-1, A=-\gamma$) with $Q(p)=\theta+\sigma((\gamma+1)p^r-\gamma p^{\gamma+1})$(for more details, see \cite{nair2012modelling}).
	
    \begin{figure}[H]
		\centering
		\includegraphics[width=0.9\linewidth]{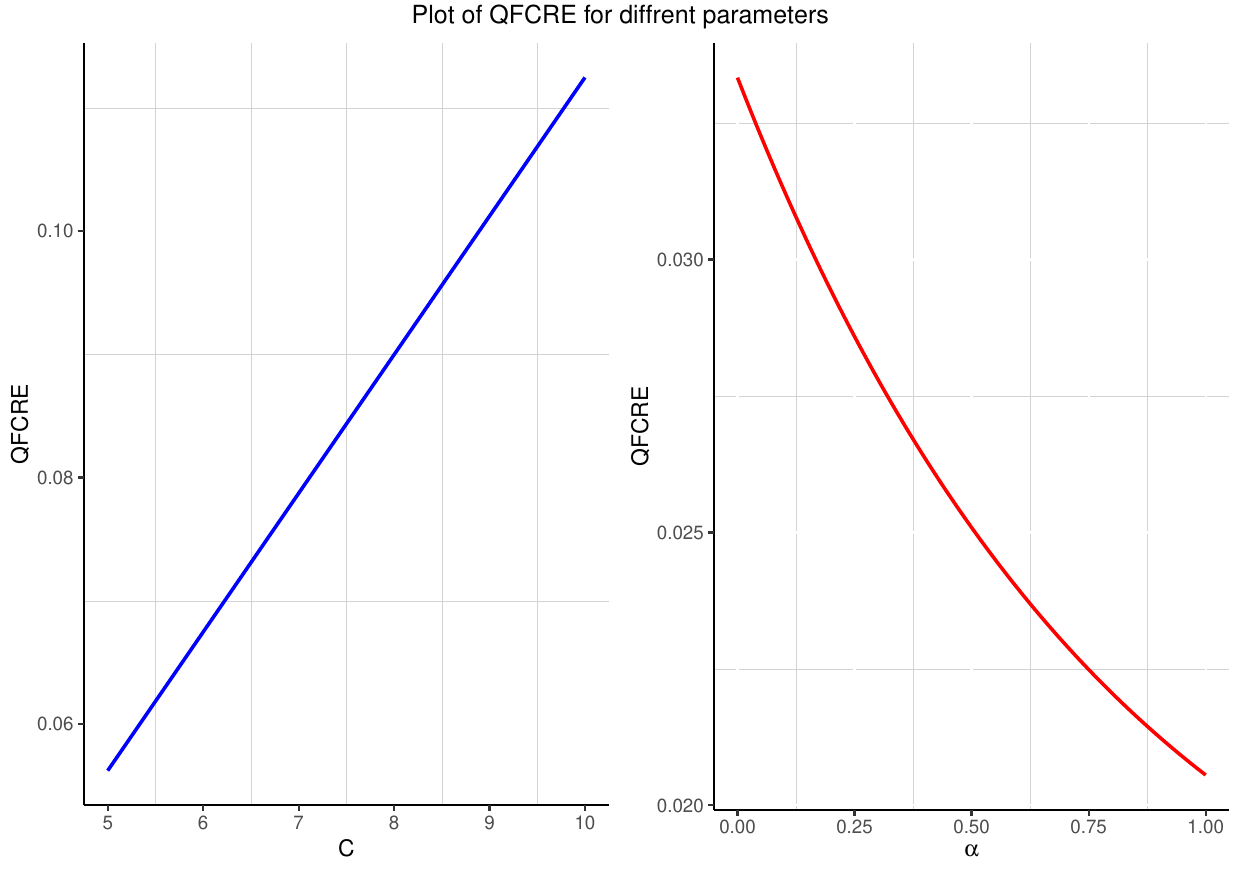}
		\caption{Plot of Q-FCRE}
		\label{fig:test1}
	\end{figure}
	\item [(iv)] Let $X$ is distributed with qdf $q(u)=(1-u)^{-A}(-\log(1-u))^{-B}$, where $A>0$, $B<1$.  Then, we have
	\begin{equation}
		\mathcal{E}_{\alpha}^{Q}(X)=\frac{ \Gamma{(\alpha-B+1)}}{(2-A)^{\alpha-B+1}}.\nonumber
	\end{equation}
	Note that $q(u)$ have  monotonic and nonmonotonic hazard quantile functions and it contains several well known distributions such as the Weibull distribution ($A=1, B=\frac{a-1}{a}$), with shape parameter $a$ and scale parameter $\lambda=Ca$ (the exponential and Rayleigh distributions are further special cases), uniform ($A=0, B=0$), Pareto ($A>1, B=0$) and rescaled beta ($A<1, B=0$).
	\item [(v)] Assume $X$ and $Y$ are non-negative rvs with QFs $Q_X(u)$ and $Q_Y(u)$ respectively. Then, $Y$ is said to be the PHM of $X$ if and only if $Q_Y(u)=Q_X(1-(1-u)^\frac{1}{\theta})$ with qdf of $Y$ $q_Y(u)=\frac{1}{\theta}q_X(1-(1-u)^\frac{1}{\theta})(1-u)^{\frac{1}{\theta}-1}$, $\theta>0$. The Q-FCRE of $Y$ is;
	\begin{equation*}
		\mathcal{E}_{\alpha}^Q(Y) = \frac{1}{\theta}\int_{0}^{1}(1-p)^{\frac{1}{\theta}}(-\log(1-p))^{\alpha}  q_X(1-(1-p)^\frac{1}{\theta}) dp.
	\end{equation*}
	For an instance, consider the qdf $q_X(u)=(1-u)^{-A}(-\log(1-u))^{-B}$  with corresponding PHM, $q_Y(u)=\theta^{B-1}(1-u)^{\frac{1-A}{\theta}-1}(-\log(1-u))^{-B}$, $A>0$, $\theta>0$, $B<1$. In this case $\mathcal{E}_{\alpha}^Q(Y)=\frac{\theta^{\alpha}\Gamma(\alpha-B+1)}{(\theta-A+1)^{\alpha-B+1}}$,  provided that $\theta > A-1$.
\end{itemize}
\subsection{Properties of Q-FCRE}
\begin{enumerate}
	\item $\mathcal{E}_{\alpha}^{Q}(X)\geq 0$, It is trivial that quantile based FCRE is always non-negative.
	\item The Q-FCRE is independent of shift parameter.  That is, $X$ possesses the same Q-FCRE as $X+b$, for an arbitrary non-negative parameter $b$.
	\item Let $a>0$, $b>0$ and $Y=aX+b$ , then $\mathcal{E}_{\alpha}^{Q}(Y)=a\mathcal{E}_{\alpha}^{Q}(X)$.
	\begin{proof}
		\begin{eqnarray}\label{3}
			\mathcal{E}_\alpha(Y)&=&\int_{0}^{\infty}\Bar{F}_Y(y)(-\log \Bar{F}_Y(y))^\alpha dy \nonumber\\
			&=& \int_{b}^{\infty} \Bar{F}_X\bigg(\frac{x-b}{a}\bigg) \bigg(-\log \Bar{F}_X\bigg(\frac{x-b}{a}\bigg)\bigg)^\alpha dx, \; x \geq b. 
		\end{eqnarray}
		By putting $\frac{x-b}{a}=z$ in \eqref{3} we get\\
		\begin{equation*}
			\mathcal{E}_{\alpha}(Y)=a \int_{0}^{\infty}\Bar{F}_X(z)(-\log \Bar{F}_X(z))^\alpha dz=a \mathcal{E}_{\alpha}(X),
		\end{equation*}
		and, $\mathcal{E}_{\alpha}^{Q}(Y)=a\int_{0}^{1}(1-p)(-\log(1-p))^\alpha q(p)dp=a \mathcal{E}_{\alpha}^{Q}(X)$.
	\end{proof}
	\item From the properties of QFs given in \cite{nair2022reliability}, if $Q(u)=Q_1(u)+Q_2(u)$ where $Q_1$ and $Q_2$ are QFs then 
	\begin{equation*}
		\mathcal{E}_{\alpha}^Q=\int_{0}^{1}(1-p)(-\log(1-p))^{\alpha}(q_1(p)+q_2(p))dp=\mathcal{E}_{\alpha}^{Q_1}+\mathcal{E}_{\alpha}^{Q_2}.
	\end{equation*}
	For example, if $Q_1(u)=u$ and $Q_2(u)=-\frac{1}{\lambda}(\log(1-u))$ denote the uniform and exponential distributions respectively,
	\begin{equation*}
		\mathcal{E}_{\alpha}^Q = \int_{0}^{1}(1-p)(-\log(1-p))^{\alpha}\bigg(\frac{1+\lambda-\lambda p}{\lambda-\lambda p}\bigg)dp.
	\end{equation*}
	Setting, $\alpha=0.5$, $\lambda=2$, $\mathcal{E}_{\alpha}^Q=0.7564$
	\item Consider two quantile functions $Q_1$ and $Q_2$ such that $Q_1(u)+Q_2(u)=Q(u)$, then 
	\begin{equation*}
		\textrm{max}\{\mathcal{E}_\alpha^{Q_1},\mathcal{E}_\alpha^{Q_2}\}\leq\mathcal{E}_\alpha^{Q_1+Q_2}.
	\end{equation*}
	The proof can be directly obtained from the above properties 1 and 4.
	\item If $Q(u)=Q_1(u)Q_2(u)$, where $Q_1$ and $Q_2$ are positive QFs, then,
	\begin{equation*}
		\mathcal{E}_{\alpha}^Q=\int_{0}^{1}(1-p)(-\log(1-p))^{\alpha}(Q_1(p)q_2(p)+Q_2(p)q_1(p))dp.
	\end{equation*}
	\item If $Q_X(u)$ is the QF of $X$. Then, the QF of $Y=\frac{1}{X}$ is $Q_Y(u)=[Q(1-u)]^{-1}$. Hence, Q-FCRE of $Y$ is
	\begin{equation*}
		\mathcal{E}_{\alpha}^Q(Y)=\int_{0}^{1}(1-p)(-\log(1-p))^{\alpha}\frac{q(1-p)}{Q^2(1-p)}dp.
	\end{equation*}
	This result can be illustrated with an example, assume $X$ follows the power distribution with $Q_X(u)=u^{\beta}$. Then, $Y=\frac{1}{X}$ has Pareto I distribution with QF, $Q_Y(u)=(1-u)^{-\beta}$. Hence, the Q-FCRE of $Y$ is $\mathcal{E}_{\alpha}^Q(Y)=\frac{\beta \Gamma(\alpha+1)}{(1-\beta)^{\alpha+1}}$.
	
	\item For a non-negaive rv $X$ and $0\leq\alpha\leq1$, it holds that $\mathcal{E}_{\alpha}^{Q}(X)\leq [\mathcal{E}^Q(X)]^\alpha$, where $\mathcal{E}^Q(X)$ is the quantile-based cumulative residual entropy (Q-CRE). 
 \begin{proof}
 Since, $(1-p)\leq (1-p)^\alpha$ for $0 \leq p, \alpha \leq 1$
		\begin{eqnarray}
			\mathcal{E}_{\alpha}^{Q}(X)&=& \int_{0}^{1}(1-p)(-\log(1-p))^\alpha q(p)dp \nonumber \\
			& \leq &\int_{0}^{1} (-(1-p)\log(1-p))^\alpha q(p)dp \nonumber\\
			& \leq & \bigg[-\int_{0}^{1}(1-p)\log(1-p)q(p)dp\bigg]^\alpha = [\mathcal{E}^Q(X)]^\alpha\nonumber
		\end{eqnarray}
		where Jensen's inequality is used to generate the final inequality. The equality holds at $\alpha = 1$.
	\end{proof}
	\item Let $X \geq 0$ with density function $f$.  Then
	$\mathcal{E}_{\alpha}^{Q}(X)\geq C(q)e^{\mathcal{E}_Q(X)}$, where $\mathcal{E}_Q(X)$ is the quantile based differential entropy and $C(q)=e^{\int_{0}^{1}\log[(1-p)(-\log(1-p))]^qdp}$ is a function of $q$.
	\begin{proof}
			Let $1-F(x)=1-p=\int_{x}^{\infty}f(t)dt$. Using the log-sum inequality,\\
		\begin{eqnarray}\label{4}
			\int_{0}^{1} f(Q(p))\log\frac{f(Q(p))}{(1-p)(-\log(1-p))^q}dQ(p) & \geq & \log \frac{1}{\int_{0}^{1} (1-p) (-\log(1-p))^q q(p)dp} \nonumber\\ 
			& = & \log \frac{1}{\mathcal{E}_{\alpha}^{Q}(X)}=-\log (\mathcal{E}_{\alpha}^{Q}(X)). 
		\end{eqnarray}
		Moreover, LHS of above inequality can be expressed as;
		\begin{equation}\label{5}
			\int_{0}^{1} f(Q(p))\log\frac{f(Q(p))}{(1-p)(-\log(1-p))^q}dQ(p)  = -\mathcal{E}_Q(X) - \int_{0}^{1} \log[(1-p)(-\log(1-p))^q] dp
		\end{equation}
		From \eqref{4} and \eqref{5},
		\begin{equation}\label{6}
			\log (\mathcal{E}_{\alpha}^{Q}(X)) \geq  \mathcal{E}_Q(X) + \int_{0}^{1} \log[(1-p)(-\log(1-p))^q] dp.   
		\end{equation}
		Exponentiating both sides of \eqref{6} we get,
		\begin{equation*}
			\mathcal{E}_{\alpha}^{Q}(X) \geq C(q)e^{\mathcal{E}_Q(X)}  
		\end{equation*}
	\end{proof}
\end{enumerate}
\begin{defn}
A rv $X$ is lesser than the rv $Y$ in Q-FCRE order, represented as $X \leq _{FCRQE} Y$, if $\mathcal{E}_{\alpha}^{Q}(X) \leq \mathcal{E}_{\alpha}^{Q}(Y)$.
\end{defn}

Now, we prove a closure property of hazard quantile order between $X$ and $Y$ and the Q-FCRE order.
\begin{thm}
If $X$ is lesser than $Y$ in the hazard quantile function order, denoted by $X\leq_{HQ}Y$, then $X \leq_{FCRQE} Y$.  
\end{thm}
\begin{proof}
	When $X\leq_{HQ}Y$ we have $H_X(u) \geq H_Y(u)$, or equivalently we obtain,
	\begin{equation} \label{eq1}
		\begin{split}
			X \leq _{HQ} Y & \Longrightarrow \frac{1}{(1-u)q_X(u)} \geq \frac{1}{(1-u)q_Y(u)} \Longrightarrow (1-u)q_X(u) \leq (1-u)q_Y(u)\nonumber \\
			& \Longrightarrow (1-u)(-\log(1-u))^{\alpha}q_X(u) \leq (1-u)(-\log(1-u))^{\alpha}q_Y(u) \\ 
			& \Longrightarrow \int_{0}^{1} (1-u)(-\log(1-u))^{\alpha}q_X(u)du \leq \int_{0}^{1} (1-u)(-\log(1-u))^{\alpha}q_Y(u)du \\ 
			& \Longrightarrow \mathcal{E}_{\alpha}^Q(X) \leq \mathcal{E}_{\alpha}^Q(Y)
		\end{split}
	\end{equation}
	as required.
\end{proof}
It is easy to show that QF hazard ordering and reversed QF hazard ordering are equivalent, $i.e.$, $X \leq_{HQ} Y \Longleftrightarrow X \geq_{RHQ} Y$. For a reversed hazard QF order definition, see Definition 2.1 of \citep{krishnan2020some}. Therefore in Theorem 1, $X \leq_{HQ} Y$ can be substituted with $X \geq_{RHQ} Y$ to
obtain $X \leq_{FCRQE} Y$.

\begin{defn}
A rv $X$ is lesser than $Y$ in terms of dispersive ordering, denoted by $X \leq_{disp} Y$, if $Q_Y (u) - Q_X (u)$ is increasing in $u \in (0, 1)$.

\end{defn}
\begin{thm}
Suppose $Q_X$ and $Q_Y$ denote the QFs of $X$ and $Y$ respectively.  We have the following results.
\begin{itemize}
	\item [(i)] If $X$ is lesser than $Y$ in the reversed hazard quantile function ordering denoted by $X \leq _{RHQ} Y$ and $\zeta$ is increasing and concave, then, $X \leq _{disp} Y \Longrightarrow \zeta(X) \geq_{FCRQE} \zeta(Y)$.
	\item [(ii)] If $X$ is lesser than $Y$ in the hazard quantile function ordering denoted by $X \leq _{HQ} Y$ and $\zeta$ is increasing and convex, then, $X \leq _{disp} Y \Longrightarrow \zeta(X) \leq_{FCRQE} \zeta(Y)$.
\end{itemize}
\end{thm}
\begin{proof}
	Case (i): Since $X \leq_{disp}Y$ implies $Q_X(u) \leq Q_{Y}(u)$ and as $\zeta$ is concave and increasing we obtain
	\begin{equation}\label{7}
		\zeta'(Q_X(u)) \geq  \zeta'(Q_Y(u)) \Longrightarrow 0 \leq \frac{1}{\zeta'(Q_X(u))} \leq \frac{1}{\zeta'(Q_Y(u))}.
	\end{equation}
	Further,
	\begin{equation}\label{8}
		X \leq_{rhq}Y \Longrightarrow \frac{1}{(1-u)q_X(u)} \leq \frac{1}{(1-u)q_Y(u)}.  
	\end{equation}
	Upon combining \eqref{7} and \eqref{8} we obtain,
	\begin{equation}
		\begin{split}
			& \Longrightarrow \frac{1}{(1-u)q_X(u)\zeta'(Q_X(u)) }\leq \frac{1}{(1-u)q_Y(u)\zeta'(Q_Y(u)) } \nonumber \\
			& \Longrightarrow \int_{0}^{1} (1-p)(-\log(1-p))^{\alpha}q_X(p)\zeta'(Q_X(p))dp \geq \int_{0}^{1} (1-p)(-\log(1-p))^{\alpha}q_Y(p)\zeta'(Q_Y(p))dp\nonumber\\
			& \Longrightarrow \mathcal{E}_{\alpha}^{Q}(\zeta(X)) \geq \mathcal{E}_{\alpha}^{Q}(\zeta(Y)).
		\end{split}
	\end{equation}
	The case (ii) can be proved analogously.
\end{proof}

The usefulness of weighted distributions can be viewed in many areas, such as reliability theory,  renewal theory, and ecology. When rv $X$ has 
pdf $f_X$, then the corresponding weighted rv $X_\omega$ has pdf of the form $f_{\omega} (x) = \frac{\omega(x)f_X(x)}{E(\omega(X))}$
, where $\omega(x)$ is a weight function having
a positive value and a finite expectation. Let us examine a specific case of 
weighted rv called an escort rv. The pdf for the escort distribution associated with $X$ is provided as
\begin{equation}\label{esc dn}
f_{X_{e,c}}(x) = \frac{f_X^c(x)}{\int_{0}^{\infty}f_X^c(x)dx}, \hspace{0.2cm} c>0, 
\end{equation}
given that the involved integral exists. Observe that a weighted distribution with a suitable weight function gives escort distribution. Now, let's examine the Q-FCRE for the escort distribution in
\eqref{esc dn}. The  density QF of $X_{e,c}$  using the QF $Q_X$ in \eqref{esc dn} is
\begin{eqnarray}\label{10}
f_{X_{e,c}}(Q_X(u))&=&\frac{f_X^c(Q_X(u))}{\int_{0}^{1}f_X^{c}(Q_X(u))dQ_X(u)}\nonumber\\
&=&\frac{f_X^c(Q_X(u))}{\int_{0}^{1}f_X^{c-1}(Q_X(p))dp}=\frac{1}{q_X^{c}(u)\int_{0}^{1}q_X^{1-c}(p)dp}=\frac{1}{q_{X_{e,c}}(u)},
\end{eqnarray}
where $q_{X_{e,c}}(u)$ is the qdf of $X_{e,c}$.
\begin{prop}
Let $X_{e,c}$ be an escort rv having qdf $q_{e,c}$ corresponding to a rv $X$ with pdf $f_X$ and qdf $q_X$. Then, Q-FCRE of $X_{e,c}$ is
\begin{equation*}
	\mathcal{E}_{\alpha}^{Q}(X_{e,c})=I_c^Q(X)\mathcal{E}_{\alpha,c}^{Q}(X)
\end{equation*}
where, $I_c^Q(X)=\int_{0}^{1} f_{X}^c(Q_X(p))dQ_X(p)$ is the quantile based information generating function(Q-IGF) and $\mathcal{E}_{\alpha,c}^{Q}(X)=\int_{0}^{1}(1-p)(-\log(1-p))^\alpha q_{X}^c(p)dp.$
\end{prop}
\begin{proof}
	Q-FCRE of $X_{e,c}$ is given by;
	\begin{eqnarray}
		\mathcal{E}_{\alpha}^{Q}(X_{e,c})&=&\int_{0}^{1}(1-p)(-\log(1-p))^\alpha q_{X_{e,c}}(p)dp\nonumber\\
		{}&=&\int_{0}^{1}(1-p)(-\log(1-p))^\alpha \frac{1}{f_{X_{e,c}}(Q_X(p))} dp. \nonumber
	\end{eqnarray}
	Now, from \eqref{8} we obtain;
	\begin{eqnarray*}
		\mathcal{E}_{\alpha}^{Q}(X_{e,c})&=&\int_{0}^{1}(1-p)(-\log(1-p))^\alpha \frac{\int_{0}^{1}f_X^{c}(Q_X(p))dQ_X(p)}{f_X^c(Q_X(p))} dp\\
		& = &  \int_{0}^{1}f_X^{c}(Q_X(p))dQ_X(p) \int_{0}^{1}(1-p)(-\log(1-p))^\alpha q_{X}^c(p)dp\\
		& = & I_c^Q(X)\mathcal{E}_{\alpha,c}^{Q}(X),
	\end{eqnarray*}
	as required.
\end{proof}

Next, we consider the montone transformation of a random variable to see its effect on quantile-based FCRE. Assuming $\zeta$ is a non-decreasing function and $Y=\zeta(X)$, where $X$ is a random variable with pdf $f_X$ and QF $Q_X$. Then, pdf of $Y$ is $f_{Y}(y)=\frac{f_X(\zeta^{-1}(y))}{\zeta'(\zeta^{-1}(y))}$. Moreover, $F_{Y}(y)=F_{X}(\zeta^{-1}(y))\Rightarrow F_Y(Q_Y(u))=F_X(\zeta^{-1}(Q_Y(u))) \Rightarrow \zeta^{-1}(Q_Y(u))=F_X^{-1}(u)=Q_X(u) $. By using this, the pdf of $Y$ can be expressed as,
\begin{equation}\label{11}
f_Y(Q_Y(u))=\frac{f_X(\zeta^{-1}(Q_Y(u))}{\zeta'(\zeta^{-1}(Q_Y(u)))}=\frac{f_X(Q_X(u))}{\zeta'(Q_X(u))}=\frac{1}{q_X(u)\zeta'(Q_X(u))}=\frac{1}{q_Y(u)}.
\end{equation}
\begin{thm}\label{three}
Suppose $Q_X$ and $q_X$ are the QF and qdf corresponding to a rv $X$. Further let $\zeta$ be a positive-valued increasing function. Then,
\begin{equation}
	\mathcal{E}_{\alpha}^{Q}(\zeta(X))=\int_{0}^{1}(1-p)(-\log(1-p))^{\alpha}q_X(p)\zeta'(Q_X(p))dp.\nonumber
\end{equation}
\end{thm}
\begin{proof}
	Let $Y=\zeta(X)$
	\begin{equation}
		\mathcal{E}_{\alpha}^Q(\zeta(X))=\int_{0}^{1}(1-p)(-\log(1-p))^{\alpha}q_{Y}(p)dp, \; 0\leq \alpha \leq 1. \nonumber
	\end{equation}
	From $\ref{11}$, we obtain $q_Y(u)=q_X(u)\zeta'(Q_X(u))$.  Hence,
	\begin{equation}
		\mathcal{E}_{\alpha}^{Q}(\zeta(X))=\int_{0}^{1}(1-p)(-\log(1-p))^\alpha q_X(p) \zeta'(Q_X(p))dp. \nonumber
	\end{equation}
\end{proof}

\begin{ex}
\textnormal{Suppose $X$ is exponential with QF,  $Q_{X}(u)=\frac{-\log(1-u)}{\theta}$, and qdf $q_X(u)=\frac{1}{\theta(1-u)}$. Using the transformation $Y = \zeta(X)=X^\beta, \beta>0$ has a Weibull distribution with QF $Q_{\zeta(X)}(u)=\theta(-\log(1-u))^\beta$, such that $\zeta(Q_X(u))=\Bigg[\frac{-\log(1-u)}{\theta}\bigg]^\beta$ and $\zeta'(Q_X(u))=\frac{\beta(-\log(1-u))^{\beta-1}}{(1-u)\theta^\beta}$. Then, Q-FCRE for the Weibull distribution is,
	\begin{equation}
		\mathcal{E}_{\alpha}^Q(\zeta(X))=\frac{\beta}{(\alpha+\beta)\theta^{\beta+1}}.
\end{equation}}
\end{ex}
\section{Quantile-based dynamic FCRE}
Many scientific disciplines, including reliability, survival analysis, actuarial science, economics, business, and many more are interested in the study of duration. Suppose $X$ represents a non-negative rv with cdf $F$ representing duration such as lifetime. In many applications, it is crucial to record the effects of the age $t$ of an individual or an item under study on the information about the residual lifetime. For example, in reliability, the study of a component's or system's lifetime after time $t$ is of great importance when a component or system of components operates at time $t$. In such situation, the residual life $\mathscr{C}_t=\{x:x>t\}$ is the set of interest. Therefore, the concept of FCRE for the residual lifetime ditribution known as dynamic fractional cumulative residual entropy (DFCRE) is required.
The dynamic fractional cumulative reidual entropy (DFCRE) function corresponding to \ref{1} is,
\begin{equation}
\mathcal{E}_{\alpha}(X,t)=\int_{t}^{\infty}\frac{\bar{F}(x)}{\bar{F}(t)}\bigg{[}-\log \frac{\bar{F}(x)}{\bar{F}(t)} \bigg{]}^\alpha dx,  0 \leq \alpha \leq 1 , t>0.
\end{equation}
The quantile based dynamic fractional cumulative reidual entropy (Q-DFCRE) function is;
\begin{equation}\label{14}
\mathcal{E}_{\alpha}^Q(X,u) = \frac{1}{1 - u}\int_{u}^{1}(1 - p) \left(\log(1-u)-\log(1-p)\right)^\alpha q(p)dp.
\end{equation}
For $0\leq \alpha \leq 1$, $\mathcal{E}_{\alpha}^Q(X,u)$  provides the spectrum of fractional information confined in the conditional survival function about the predictability of an outcome of $X$ until $100(1-u)\%$ point of distribution. Further, when $u \rightarrow 0$, the Q-DFCRE becomes Q-FCRE.
\eqref{14} can be represented in terms of hazard quantile function as;
\begin{equation}
\mathcal{E}_{\alpha}^Q(X,u)=\frac{1}{1-u}\int_{u}^{1}[H(p)]^{-1}\left(\log(1-u)-\log(1-p)\right)^\alpha dp.
\end{equation}

\begin{ex}
\begin{enumerate}
\item For exponential distribution with hazard rate $\lambda$, $\mathcal{E}_{\alpha}^Q(X,u)=\frac{\Gamma(\alpha+1)}{\lambda}$.
\item For power distribution with QF $Q(u)=\beta u ^{\delta}$, $\beta>0, \delta>0$, we obtain $\mathcal{E}_{\alpha}^Q(X,u)=\frac{\beta\delta}{1-u}\int_{u}^{1}(1-p)(\log(1-u)-\log(1-p))^{\alpha} p^{\delta-1}dp$.
\item When $X$ follows the re-scaled beta distribution with QF $Q(u)=c[1-(1-u)^{\frac{1}{r}}]$, $r,c>0$, then $\mathcal{E}_{\alpha}^Q(X,u)=\frac{r^{\alpha}c(1-u)^\frac{1}{r}\Gamma(\alpha+1)}{(r+1)^{\alpha+1}}$.
\begin{figure}[H]
	\centering
	\includegraphics[width=0.9\linewidth]{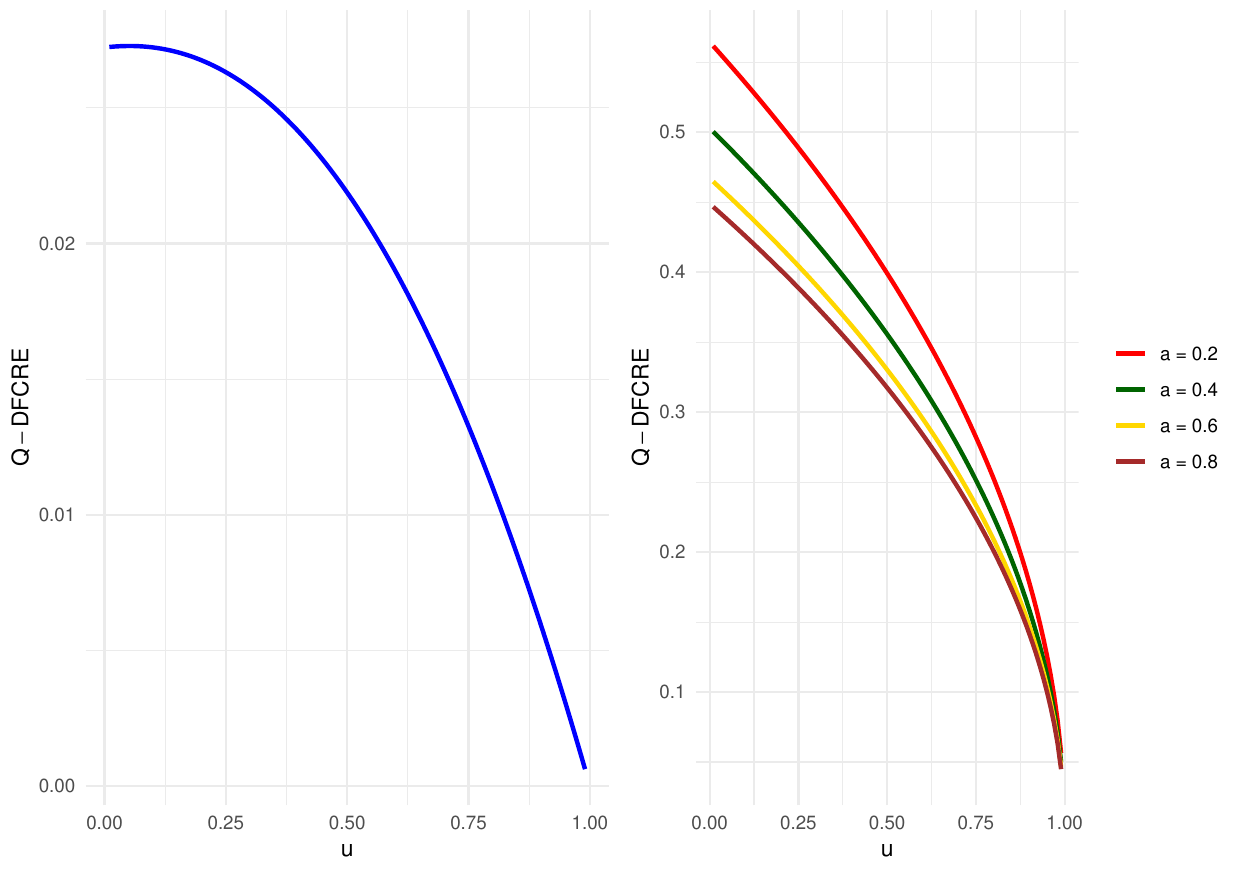}
	\caption{\textit{(a) Plot of Q-DFCRE for power distribution given in Example 2 for $\beta= 0.1, \delta=2.3, 
			\alpha=0.75$  (b) Graph of Q-DFCRE 
			for re-scaled beta distribution given in Example 3 for $c=2, r=2, \alpha=0.2, 0.4, 0.6, 0.8$. Here, we have taken the values of $u \in (0, 1)$ along the x-axis. }}
		\label{fig3}
\end{figure}
\end{enumerate}
\end{ex}

\begin{ex}
For the rv with linear mean residual QF, $Q(u)=-(a+b)\ln(1-u)-4u$ \citep{midhu2013class} that has no tractable distribution function, the corresponding Q-DFCRE becomes $\mathcal{E}_{\alpha}^{Q}(u)=(1-u)\Gamma(\alpha+1)\bigg[a+b-\frac{4}{2^{\alpha+1}}\bigg] $, where $0\leq\alpha\leq1$.
\end{ex}

Next, we introduce two nonparametric classes of life distributions based on Q-DFCRE.

\begin{defn}
A rv $X$ is having increasing (decreasing)  Q-DFCRE, $i.e.$, 
IQ-DFCRE (DQ-DFCRE) if $\mathcal{E}_{\alpha}^Q(u)$ is  increasing (decreasing) in $u \geq 0$. 
\end{defn}

Clealy, $\mathcal{E}_{\alpha}^{Q'}(u)\geq (\leq)0 \Longrightarrow \mathcal{E}_{\alpha}^Q(u) \geq (\leq)\frac{\alpha}{1-u}\int_{u}^{1}(\log (\frac{1-u}{1-p}))^{\alpha-1}dp$ according as $X$ is IQ-DFCRE (DQ-DFCRE).  For exponential distribution, $\mathcal{E}_{\alpha}^Q(u)$ is independent of $u$, implying that this distribution is the boundary of the IQ-DFCRE (DQ-DFCRE) classes.

\begin{defn}
The rv $X$ has less dynamic fractional cumulative residual quantile entropy than the random variable $Y$ if, $\mathcal{E}_{\alpha}^{Q}(X,u) \leq \mathcal{E}_{\alpha}^{Q}(Y,u) $. We write $X \leq_{QDFCRE}Y$.
\end{defn}

Suppose $X$ and $Y$ be two exponentially distributed random variables with hazard rates $\lambda_1$ and $\lambda_2$.  Then $X\leq_{QDFCRE}Y$ for $\lambda_1\geq\lambda_2$.
\begin{thm}
Let $Q_X$ and $q_X$ respectively denote the QF and qdf of a rv $X$ and let $\zeta$ is a positive-valued increasing function. Then,
\begin{equation*}
	\mathcal{E}_{\alpha}^Q(\zeta(X), u) = \frac{1}{1-u}\int_{u}^{1}(1-p)\left(\log(1-u)-\log(1-p)\right)^\alpha q_X(p) \zeta'(Q_X(p))dp.
\end{equation*}
\end{thm}
\begin{proof}
	The proof is similar to that of the proof of the Theorem \ref{three}.
\end{proof}

\begin{thm}
For a non-negative rv $X$, consider a transformation $Y=\zeta(X)$, where $\zeta(.)$ is positive, real-valued, increasing and convex (concave) function.  Then, for $0<\alpha<1$, $Y$ is IQ-DFCRE (DQ-DFCRE) according as $X$ is IQ-DFCRE (DQ-DFCRE).
\end{thm}
\begin{proof}
	By virtue of the defintion of Q-DFCRE, we write
	\begin{eqnarray}
		\mathcal{E}_{\alpha}^Q(Y,u)&=&\frac{1}{1-u}\int_{u}^{1}(1-p)(\log(1-u)-\log(1-p))^\alpha q_{Y}(p)dp \nonumber\\
		& = & \frac{1}{1-u}\int_{u}^{1}(1-p)(\log(1-u)-\log(1-p))^\alpha q_{X}(p) \zeta'(Q_{X}(p))dp.
	\end{eqnarray}
	Since, $\zeta$ is nonnegative, increasing and convex (concave) we have $\zeta'(Q_{X}(u))$ is increasing (decreasing) and nonnegative. Hence by Lemma 3.1 of \cite{nanda2014renyi} $\mathcal{E}_{\alpha}^Q(Y,u)$ is increasing (decreasing) in $u$. Hence, for $0<\alpha<1$, $\mathcal{E}_{\alpha}^Q(Y,u)$ is increasing (decreasing) in $u$ whenever, $\mathcal{E}_{\alpha}^{Q}(X,u)$ is increasing (decreasing) in $u$.
\end{proof}

\begin{cor}
Let $Y=aX+b$, $a>0$, $b \geq 0$. If $X$ is IQ-DFCRE (DQ-DFCRE), then $Y$ is also IQ-DFCRE (DQ-DFCRE).
\end{cor}

\begin{lem}
The stochastic ordering in quantile failure rate ($H_{X}(u)\geq(\leq)H_{Y}(u)$) does not preserve Q-DFCRE ordering.
\end{lem} 

For example, consider $X \sim U(0,2)$ with $Q_X(u)=2u$, $H_{X}(u)=\frac{1}{2(1-u)}$, and $\mathcal{E}_{\alpha}^{Q}(X,u)=(1-u)\frac{\Gamma(\alpha+1)}{2^\alpha}$; and $Y \sim Exp(1)$ with $Q_Y(u)=-\ln(1-u)$, $H_{Y}(u)=1$ such that $\mathcal{E}_{\alpha}^{Q}(Y,u)=\Gamma(\alpha+1)$. Clearly, $\mathcal{E}_{\alpha}^{Q}(X,u)\leq \mathcal{E}_{\alpha}^{Q}(Y,u)$ while $H_{X}(u)\geq(\leq)H_{Y}(u)$ for $\frac{1}{2} < u < 1(0<u<\frac{1}{2})$.\\

Next, we obtain a relationship between Q-DFCRE and dispersive orderings.
\begin{thm}
For the non-negative rvs $X$ and $Y$, $X \leq_{disp} Y \Longrightarrow X \leq_{QDFCRE}Y$.  
\end{thm}
\begin{proof}
	$X \leq_{disp} Y$ implies $Q_Y(u)-Q_X(u)$ is increasing with respect to $u \in (0,1)$ implying that $q_X(u) \leq q_Y(u)$. Therefore, we get
	\begin{eqnarray*}
		\frac{1}{1-u}&&\int_{u}^{1}(1-p)\left(\log(1-u)-\log(1-p)\right)^{\alpha}q_{X}(p)dp\\&&\leq  \frac{1}{1-u}\int_{u}^{1}(1-p)\left(\log(1-u)-\log(1-p)\right)^{\alpha}q_{Y}(p)dp.  
	\end{eqnarray*}
\end{proof}

\begin{thm}
Let $X$ and $Y$ be two rvs such that $X \leq _{QDFCRE} Y$. Then for a non-negative increasing convex function $\zeta$, $\zeta(X)\leq_{QDFCRE} \zeta(Y)$.
\end{thm}
\begin{proof}
	It is enough to show that, $\forall$ $ u \in [0,1]$ 
	\begin{eqnarray}\label{17}
		\frac{1}{1-u}&&\int_{u}^{1}[(1-p)(\log(1-u)-\log(1-p))^\alpha q_{Y}(p) \phi'(Q_{Y}(p))]dp \nonumber \\
		&&\geq \frac{1}{1-u}\int_{u}^{1}[(1-p)(\log(1-u)-\log(1-p))^\alpha q_{X}(p) \zeta'(Q_{X}(p))]dp.
	\end{eqnarray}
	Since, $X \leq _{FCRQE}Y$, we have $\forall$ $u \in [0,1]$
	\begin{eqnarray}\label{18}
		\frac{1}{1-u}&&\int_{u}^{1}[(1-p)(\log(1-u)-\log(1-p))^\alpha q_{Y}(p)]dp  \nonumber
		\\  
		&&\geq  \frac{1}{1-u}\int_{u}^{1}[(1-p)(\log(1-u)-\log(1-p))^\alpha q_{X}(p)]dp \nonumber \\
		&&\Longrightarrow \int_{u}^{1}[(1-p)[\log(1-u)-\log(1-p)]^\alpha[q_{Y}(p)-q_{X}(p)]]dp\geq 0
	\end{eqnarray}
	Here, $(1-p)$ is always positive. $\log(1-u)-\log(1-p) >0 \; \text{for} \; p > u $ and $q_{Y}(p)-q_{X}(p)>0 \Longrightarrow q_{Y}(p)>q_{X}(p) \Longrightarrow Q_{Y}(p)>Q_{X}(p)\Longrightarrow \zeta'(Q_{Y}(p))>\zeta'(Q_{X}(p))$, proves the result. 
\end{proof}

\section{Nonparametric estimation of Q-FCRE}

In this section, a nonparametric estimator for Q-FCRE is constructed. Let us consider the lifetime of $n$ iid components represented as $X_1, X_2,\ldots, X_n$  with a common cdf $F(x)$ and QF $Q(u)$. Let $X_{1:n},X_{2:n},\ldots,X_{n:n}$ denote the respective order statistics. \cite{parzen1979nonparametric} introduced an empirical quantile function
\begin{equation*}
\bar{Q}(u)=X_{k:n},\hspace{0.2cm}\frac{k-1}{n}<u<\frac{k}{n},\hspace{0.2cm}k=1,2,\ldots,n,
\end{equation*}
which is a step function with jump $\frac{1}{n}$. A smoothed version of this estimator is given by,
\begin{equation*}
\bar{Q}_n(u)=n\bigg(\frac{k}{n}-u\bigg)X_{k-1:n}+n\bigg(u-\frac{k-1}{n}\bigg)X_{k:n}
\end{equation*}
for $\frac{k-1}{n}<u<\frac{k}{n}$, $k=1,2,\ldots,n$ (see \cite{ parzen1979nonparametric}). The corresponding estimator of the empirical quantile density function is
\begin{equation}\label{19}
\bar{q}_n{(u)}=\frac{d}{du}\bar{Q}_n(u)=n(X_{k:n}-X_{k-1:n}), \: \textrm{for} \:\frac{k-1}{n}<u<\frac{k}{n}.
\end{equation}
Using \eqref{19}, the non parametric estimator of Q-FCRE becomes
\begin{equation}\label{20}
\hat{\mathcal{E}}_{\alpha}^Q(X)=\int_{0}^ {1}(1-p)(-\log(1-p))^\alpha \bar{q}(p)dp.
\end{equation}
Now, by approximating integral to summation in \eqref{20}, the plug-in estimator of Q-FCRE is
\begin{equation}\label{21}
\hat{\mathcal{E}}_{\alpha}^Q(X)=\sum_{i=1}^{n}(1-\hat{F}(X_{i:n}))(-\log(1-\hat{F}(X_{i:n})))^{\alpha}n(X_{i:n}-X_{i-1:n})(S_{i:n}-S_{i-1:n})
\end{equation}
where $\hat{F}(X_{i:n})$ is the empirical distribution function and $S_{i:n}$ is defined by
\[S_{i:n}
=
\begin{cases}
0, i=0  \\
F(X_{i:n})=\frac{i}{n}, i=1,2,3,\ldots,n-1 \\
1,i=n
\end{cases}
\]
Hence, the simplified expression of \eqref{21} is
\begin{equation}\label{fcre est}
\hat{\mathcal{E}}_{\alpha}^Q(X)= \sum_{i=1}^{n-1}\bigg(1-\frac{i}{n}\bigg)\bigg[-\log\bigg(1-\frac{i}{n}\bigg)\bigg]^\alpha(X_{i:n}-X_{i-1:n})
\end{equation}

\section{Simulation study}
Here, simulation studies are carried out to verify the performance of the non-parametric estimator $\hat{\mathcal{E}}_{\alpha}^Q(X)$ in $\eqref{fcre est}$. We have generated random samples of different
sizes from the power-Pareto distribution with quantile function
\begin{equation}\label{dvs dn}
Q(u)=Cu^{\lambda_1}(1-u)^{-\lambda_2}, C,\lambda_1,\lambda_2 >0.
\end{equation}
The model \eqref{dvs dn} also known as Davies distribution \citep{hankin2006new}, which is a flexible family
for right-skewed non-negative data giving a good approximation to the exponential, gamma, Weibull
and lognormal distributions. Further, when $\lambda_1 = 
\lambda_2 = \lambda$, it reduces to the log-logistic distribution. We have computed the estimated values of the Q-FCRE along with its, bias, and MSE for $\alpha=0.25$ and $\alpha=0.5$ (see Tables \ref{Table 1} and \ref{Table 2}). It is clear that both the bias and MSE decrease with an increase in the sample size $n$, validating the aymptotic property of the estimator.

\begin{table}[H]
\centering
\caption{Bias and MSE of $\hat{\mathcal{E}}_{\alpha}^Q$ for power-Pareto distribution with $C=1.5,\lambda_1=2,\lambda_2=0.25$ at $\alpha=0.25$ and $\mathcal{E}_{\alpha}^Q=0.8235$}
\begin{tabular}{@{}p{3.5cm} p{3.5cm} p{3.5cm} p{3cm}@{}} 
	\toprule
	n & $\hat{\mathcal{E}}_{\alpha}^Q$ & Bias & MSE \\ \midrule
	50 & 0.7136 & -0.1101 & 0.0275 \\
	100 & 0.7564 & -0.0671 & 0.0130\\
	250 &  0.7914 & -0.0321 & 0.0048 \\
	500 &  0.8042 & -0.0194 & 0.0023 \\
	1000 & 0.8133 &-0.0102 & 0.0011\\
	\bottomrule
\end{tabular}
\label{Table 1}
\end{table}

\begin{table}[H]
\centering
\caption{Bias and MSE of $\hat{\mathcal{E}}_{\alpha}^Q$ for power-Pareto distribution with $C=1.5,\lambda_1=2,\lambda_2=0.25$ at $\alpha=0.5$ and $\mathcal{E}_{\alpha}^Q=0.8548$}
\begin{tabular}{@{}p{3.5cm} p{3.5cm} p{3.5cm} p{3cm}@{}} 
	\toprule
	n & $\hat{\mathcal{E}}_{\alpha}^Q$ & Bias & MSE \\ \midrule
	50 & 0.7203 &-0.1345 & 0.0343  \\
	100 & 0.7699 &-0.0848 &0.0164  \\
	250    & 0.8126 &-0.0422 &0.0061   \\
	500    & 0.8293 & -0.0254 & 0.0030 \\
	1000   & 0.8413 &-0.0135 & 0.0014 \\ \bottomrule
\end{tabular}
\label{Table 2}
\end{table}

To evaluate the further performance of the estimator we consider another quantile model namely the Govindarajulu $(\theta, \sigma, \beta)$ distribution that does not have a closed expression of distribution function but has a QF with the form
\begin{equation}\label{gov dn}
Q(u) = \theta +\sigma((\beta+1)u^{\beta}-\beta u^{\beta+1}),\hspace{0.2cm} \theta \in (-\infty,+\infty), \sigma, \beta > 0, \hspace{0.2cm} 0 \leq u \leq 1.
\end{equation}
The Govindarajulu distribution \eqref{gov dn} is a significant quantile model primarily used to depict
bathtub-shaped lifetime data. For a detailed study of \eqref{gov dn} we refer to \cite{nair2013quantile}. The bias and MSE of $\hat{\mathcal{E}}_\alpha^Q$ with respect to the true parameter values
of \eqref{gov dn} are given in Table \ref{Table 3} and Table \ref{Table 4}.
\begin{table}[H]
\centering
\caption{Bias and MSE of $\hat{\mathcal{E}}_{\alpha}^Q$ using Govindarajulu distribution with $\theta=1, \sigma=2, \beta=2 $ at $\alpha=0.75$ and $\mathcal{E}_{\alpha}^Q=0.6380$}
\begin{tabular}{@{}p{3.5cm} p{3.5cm} p{3.5cm} p{3cm}@{}} 
	\toprule
	n & $\hat{\mathcal{E}}_{\alpha}^Q$ & Bias & MSE \\ \midrule
	75 & 0.6693 & 0.0314 & $1.4127 \times 10^{-3}$\\
	100 & 0.6638 & 0.0258 & $9.8400 \times 10^{-4}$\\
	250 & 0.6517  & 0.0138 & $3.1394 \times 10^{-4}$ \\
	500 & 0.6465 & 0.0086 & $1.3198\times 10 ^{-4}$ \\
	1000 &  0.6432 & 0.0053 & $5.7120 \times 10^{-5}$\\
	\bottomrule
\end{tabular}
\label{Table 3}
\end{table}

\begin{table}[H]
\centering
\caption{Bias and MSE of $\hat{\mathcal{E}}_{\alpha}^Q$ using Govindarajulu distribution with $\theta=1, \sigma=2, \beta=2 $ at $\alpha=0.85$ and $\mathcal{E}_{\alpha}^Q=0.6135$}
\begin{tabular}{@{}p{3.5cm} p{3.5cm} p{3.5cm} p{3cm}@{}} 
	\toprule
	n & $\hat{\mathcal{E}}_{\alpha}^Q$ & Bias & MSE \\ \midrule
	75 & 0.6334 & 0.0199 & $8.5742\times 10^{-4}$  \\
	100 &  0.6293 & 0.0158 & $6.0706 \times 10^{-4}$  \\
	250    & 0.6211 & 0.0075 & $1.9553 \times 10^{-4}$   \\
	500    & 0.6180 & 0.0045 & $8.4540 \times 10^{-5}$ \\
	1000   & 0.6163 & 0.0027 & $4.1717 \times 10^{-5}$ \\ \bottomrule
\end{tabular}
\label{Table 4}
\end{table}
From Tables \ref{Table 3} and \ref{Table 4}, it is clear that both bias and MSE decrease with an increase in sample size, which further confirms the performance of the estimator $\hat{\mathcal{E}}_{\alpha}^Q$ based on Govindarajulu (1,2,2) model for fractional parameters $\alpha=0.75$ and $\alpha=0.85$.
\section{Validity of Q-FCRE}
In this section, we test the reliability of the proposed measure by conducting simulations on logistic map (\cite{may1976simple}), which is a mathematical model defined by the recurrence relation:
\begin{equation}
x_{t+1}=ax_t(1-x_t), 
\end{equation}
where $x \in [0,1]$ and $a$ is a parameter that controls the dynamics of the map known as control parameter which lies in  [0,4]. For $a > 3$ the system shows chaotic behaviour and for $a<3$ it shows periodic, stable behaviour. Here, we had taken the initial value of $x_0$ as 0.1. The parameter $a$ is chosen as 1, 1.5, 2, 2.5, 3, 3.5 and 4.
\begin{figure}[H]
\centering
\includegraphics[width=0.7\linewidth]{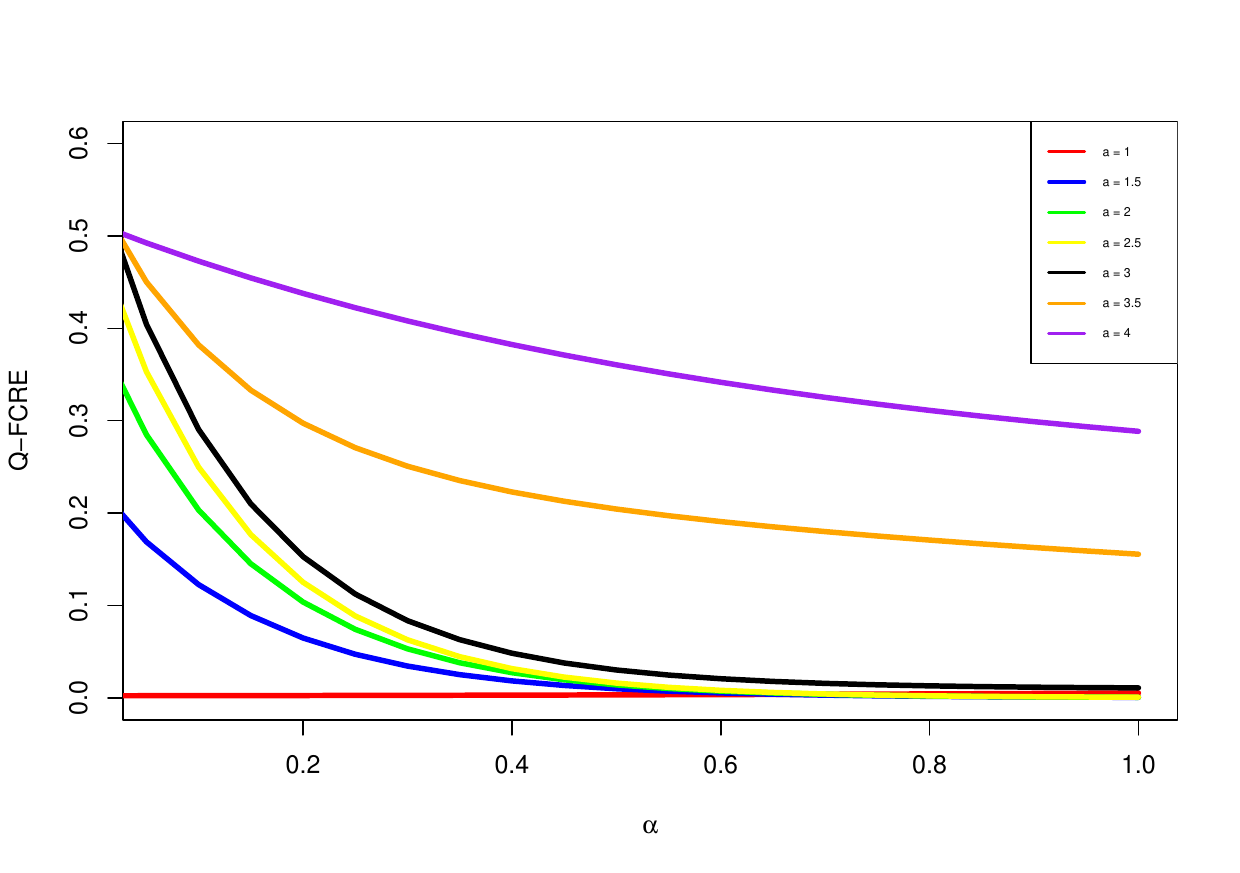}
\caption{\textit{Q-FCRE of logistic map with varied parameters $a$, for the sample size $n=2000$.}}  
\label{fig4}
\end{figure}

It can be seen from Figure \ref{fig4} that the Q-FCRE of logistic map can properly characterize the difference of uncertainty between chaotic and periodic series. For the logistic map exhibiting chaotic
behavior ($i.e.$, $a = 3, 3.5$ and $4$), they show higher entropy values than periodic ones. Moreover, with increasing parameter $a$ the degree of irregularity increases. As expected, the higher values of the logistic map parameter $a$ correspond to more chaotic behavior, which results in higher entropy. Lower values of $a$ (such as 1 and 1.5) correspond to stable or periodic states with lower entropy.   Thus, the results obtained have demonstrated that the Q-FCRE is a valid measure of uncertainty, 
showing its potential in applications of complex systems that are chaotic in nature.

\section{Applications of Q-FCRE in the financial data}

This section applies Q-FCRE to the price returns of Dow Jones Industrial Average (DJIA) during the period from January 1, 2014 to December 31, 2019 (\cite{Yahoo}). The price return is defined as $r_t=\log (x_t)-\log (x_{t-1})$,  where $x_t$ is the closing price on a trading day $t$ . The price return $r_t$ is then transformed by $y_t=r_t-\min\{r_t\}_{t=1}^{T}$ to make the data values non-negative (Figure \ref{fig5}) and further calculate the Q-FCRE value (Figure \ref{fig6}).

\begin{figure}[H]
\centering
\includegraphics[width=0.75\linewidth]{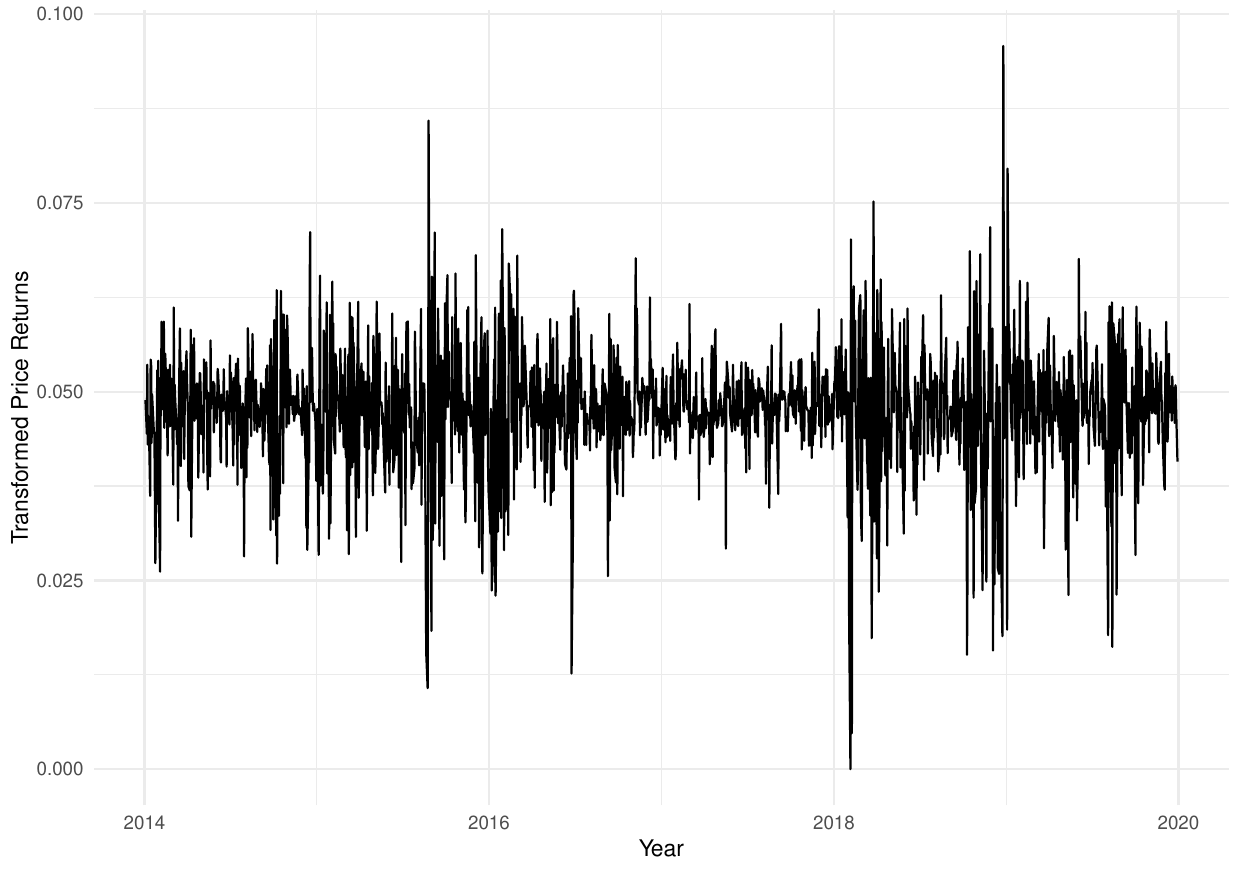}
\caption{\textit{Transformed price returns} }
\label{fig5}
\end{figure}
\begin{figure}[H]
\centering
\includegraphics[width=0.75\linewidth]{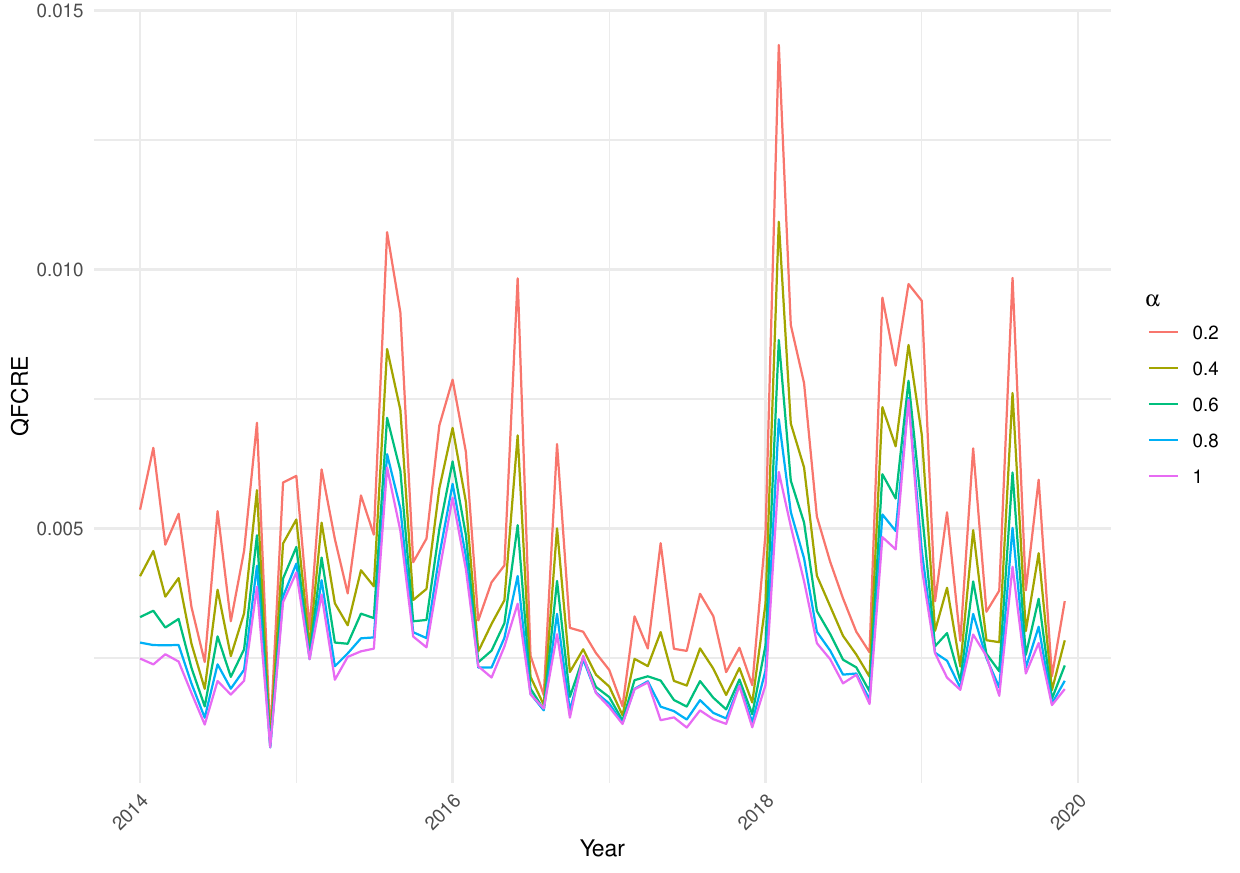}
\caption{\textit{Q-FCRE of DJIA dataset during the year 2014-2016} }
\label{fig6}
\end{figure}
Figure \ref{fig6} depicts the Q-FCRE of DJIA dataset with respect to various $\alpha$ values, $i.e.$, $\alpha=0.2, 0.4, 0.6, 0.8, 1$ during the time period 2014 to 2019.  As we can see, the Q-FCRE has a high sensitivity for $\alpha < 0.5$. On the other hand, when $\alpha$ is close to 1, the Q-FCRE decreases and, therefore, does not detect the financial instability in the data. Therefore, for the classic quantile-based CRE ($i.e.$, when $\alpha=1$), it is unable to provide as much information about the financial system as compared to the quantile based fractional CRE. From this perspective, we can consider that the Q-FCRE is superior to the classic Q-CRE. Considering the fractional parameter $\alpha$, the Q-FCRE can retrieve more inherent information carried by the underlying system, and thus detect dynamics more accurately.\\

Considering the DJIA data for the six years (2014-2019) in Figure \ref{fig6} also allow us to conclude that the financial instability/complexity in DJIA is more during the year 2018.

\begin{figure}[H]
\centering
\includegraphics[width=0.9\linewidth]{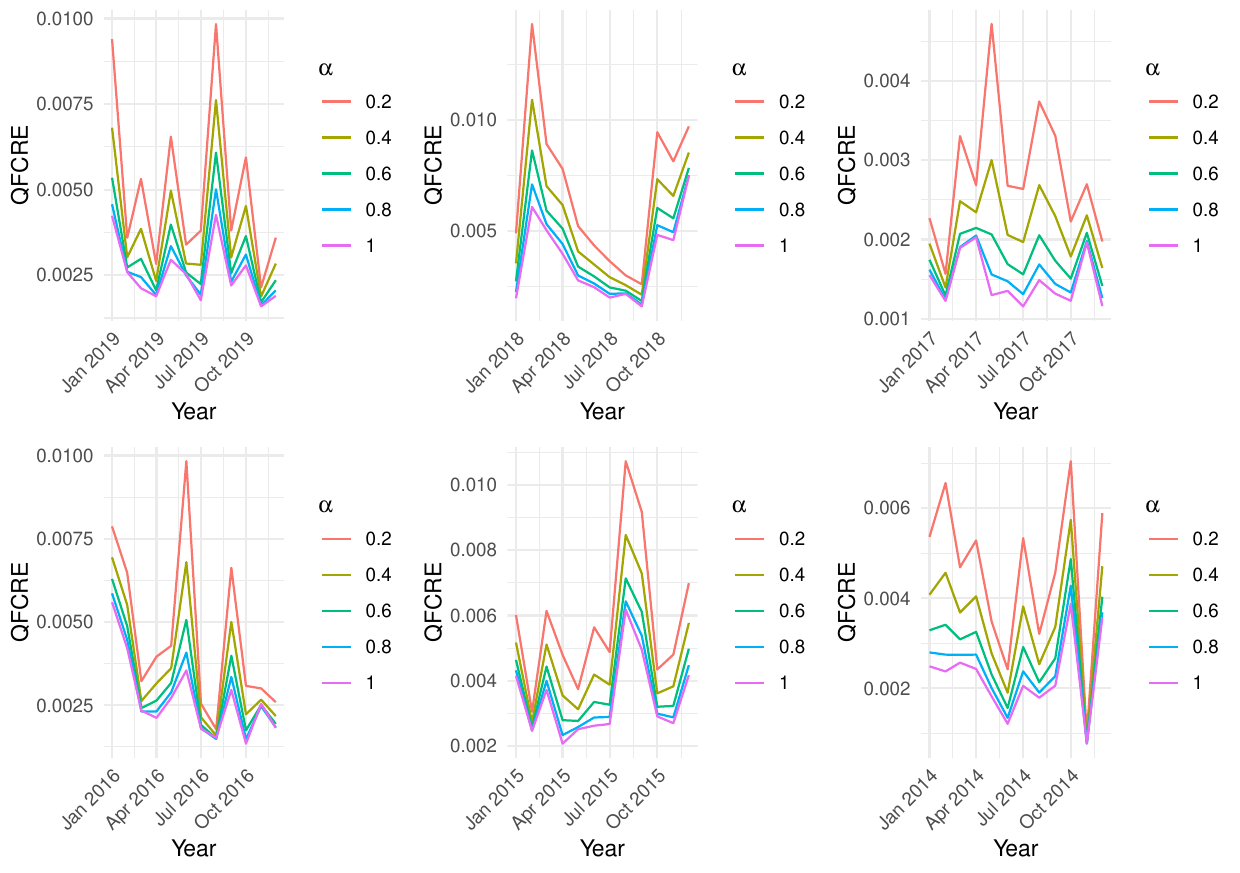}
\caption{\textit{Q-FCRE of DJIA dataset} }
\label{fig7}
\end{figure}
From Figure \ref{fig7}, it can be further elucidate that during the year 2018 the value of Q-FCRE transcended the value 0.010 for $\alpha=0.2$, hence showing high complexity in the financial data. For 2015 also the Q-FCRE value goes beyond 0.01 but does not show much complexity as compared to 2018.  However, for the rest of the years, the Q-FCRE is less than 0.01. Therefore, we can conclude that taking the years 2014-2019 into account, the DJIA data shows high financial instability during the year 2018.

\section{Concluding Remarks}
Quantile-based Fractional Cumulative Residual Entropy (Q-FCRE) provides a more dynamic, robust, and flexible measure of uncertainty, particularly for complex and evolving systems such as economic and financial markets. It excels in capturing long-term memory, nonstationary behavior, and persistent volatility, making it far more suitable for analyzing economic crises or systems with long-range dependencies than the more static and simplified Quantile-based Shannon Entropy (Q-SE).  Unlike the FCRE in the distributional framework, its quantile framework allows us to use some unique quantile models that do not have any closed-form distribution function.  The study obtained new theoretical results for Q-FCRE and its dynamic version.  The computation of the Q-FCRE for numerical data have been studied by proposing a nonparametric estimator, and validity of the estimator has also been evaluated using simulation studies.  The utility of Q-FCRE in analyzing complex financial data has been examined using the Dow Jones Industrial Average (DJIA) dataset. 

\section*{Data Availability Statement}
The data that support the findings of this study are openly available in \cite{Yahoo}.

\section*{Research ethics statement}
The data used in this study were obtained from the openly available in \cite{Yahoo} and ethical approval was not required.

\section*{Conflict of interest statement}
On behalf of all authors, the corresponding author states that there is no conflict of interest.

\section*{Acknowledgement}
The financial support from the INSPIRE, Department of Science and Technology, Government of India is gratefully acknowledged. 

\bibliographystyle{apalike}
\bibliography{ref} 

\end{document}